\documentclass[a4paper,12pt]{amsart}
\usepackage{amsfonts}
\usepackage{color}
\usepackage{url}
\usepackage{bigints}
\usepackage[normalem]{ulem}
\usepackage{graphicx}
\usepackage{amsmath,amssymb,amsthm,amsfonts}
\usepackage{amssymb}
\usepackage[active]{srcltx}
\usepackage{graphics}
\usepackage{graphicx}
\usepackage[active]{srcltx}
\usepackage[hyperpageref]{backref}
\usepackage[colorlinks=true,urlcolor=blue,citecolor=red,linkcolor=blue, linktocpage,pdfpagelabels,bookmarksnumbered,bookmarksopen]{hyperref}

\newtheorem{thm}{Theorem}[section]

\newtheorem{lem}[thm]{Lemma}
\newtheorem{prop}[thm]{Proposition}

\newtheorem{defn}[thm]{Definition}
\newcommand{\R}{\mathbb{R}}


\begin{document}

\parindent 0pc
\parskip 6pt
\overfullrule=0pt

\title[Monotonicity of positive solutions]{Monotonicity of positive solutions to quasilinear elliptic equations in half-spaces with a changing-sign nonlinearity}

\author{Francesco Esposito*, Alberto Farina$^+$, Luigi Montoro* and Berardino Sciunzi*}

\thanks{F. Esposito, L. Montoro and B. Sciunzi were partially supported by PRIN project  2017JPCAPN (Italy): {\em Qualitative and quantitative aspects of nonlinear PDEs. F. Esposito è titolare di un Assegno di Ricerca dell’Istituto Nazionale di Alta Matematica.}}

\address{*Dipartimento di Matematica e Informatica, UNICAL,
Ponte Pietro  Bucci 31B, 87036 Arcavacata di Rende, Cosenza, Italy.}

\email{francesco.esposito@unical.it}

\email{montoro@mat.unical.it}

\email{sciunzi@mat.unical.it}

\address{+ Universit\'e de Picardie Jules Verne, LAMFA, CNRS UMR
7352, 33, rue Saint-Leu 80039 Amiens, France}

\email{alberto.farina@u-picardie.fr}


\begin{abstract}
In this paper we prove the monotonicity of positive solutions to $ -\Delta_p u = f(u) $ in half-spaces under zero Dirichlet boundary conditions, for $(2N+2)/(N+2) < p < 2$ and for a general class of regular changing-sign nonlinearities $f$. The techniques used in the proof of the main result are based on a fine use of comparison and maximum principles and on an adaptation of the celebrated moving plane method to quasilinear elliptic equations in unbounded domains.
\end{abstract}

\maketitle

\section{Introduction and statement of the main result}
In this paper we deal with the study of the monotonicity of weak positive solutions of class $\mathcal{C}^1$ to the following quasilinear elliptic problem
\begin{equation}\label{equation1}
\begin{cases}
-\Delta_p u = f(u), \quad & \text{in} \, \, \R^N_+,\\
u(x',y)>0, \quad & \text{in} \, \, \R^N_+,\\
u(x',0)=0, \quad & \text{on} \, \, \partial \R^N_+
\end{cases}
\end{equation}
where $p>1$, $ N>1$ and $\mathbb{R}^N_+:=\{x=(x_1,x_2, \ldots, x_{N}) \in \R^N \ | \ x_N>0\}$ and $\Delta_p u:=\operatorname {div} (|\nabla u|^{p-2}\nabla u)$ is the usual $p$-Laplace operator.
In what follows, we shall denote a generic point of $\R^N_+$ by $(x',y)$ with $x'=(x_1,x_2, \ldots, x_{N-1}) \in \R^{N-1}$ and $y=x_N \in (0,+\infty)$.
As well known, according to the regularity results in \cite{DB, T}, solutions of equations involving the $p$-Laplace operator are, in general, only continuously differentiable
and, therefore, the equation \eqref{equation1} has to be understood in the weak sense, see Definition \ref{weakform} below.

The study of monotonicity properties of the solutions is an important task that sometimes appears in many applications such as blow-up analysis, a-priori estimates and also in the proof of Liouville type theorems. Qualitative properties of the solutions to \eqref{equation1} were firstly studied in the semilinear case $p=2$, in a series of papers due to H. Berestycki, L. Caffarelli and L. Nirenberg  \cite{BCN0,BCN1,BCN2}. More recent results for the case $p=2$ can be found in the works \cite{DupaigneFarina, Farina1,farsciu, farsciubis}. We also point out \cite{charro, QS} for results concerning the monotonicity of solutions in half-spaces for equations involving a fully nonlinear uniformly elliptic operator.
The situation in the quasilinear case is much more involved. Indeed, when working with the singular or degenerate $p$-Laplace operator, both the weak and the strong comparison principles might fail. The possible loss of these crucial tools being due either by the presence of critical points, or by the fact that the non-linearity $f$ changes sign. Added to this is the fact that we are facing a problem on an unbounded domain. Also, when $p \neq 2$ we cannot exploit the usual arguments and tricks related to the linear character of the  Laplace operator.

The present work gives new results regarding the monotonicity of positive solutions in half-spaces. Here, we continue the study of qualitative properties in half-spaces started in previous papers. In particular, in \cite{FMS}, the authors assumed that $(2N+2)/(N+2)<p<2$ with the nonlinearity $f$ a positive locally Lipschitz continuous function. Later on, in \cite{FMSR}, the same results was extended for every $p \in (1,2)$ and to equations involving a first order term. A partial answer in the degenerate case $p>2$ was obtained in \cite{FMS3}, where power-like nonlinearities were considered under the restriction $2<p<3$. Some years later in \cite{FMS_ans}, the authors removed this restriction and the monotonicity result was extended to a class of positive nonlinearities that are superlinear at zero.
We also mention the works \cite{DanDuEf, DuGuo}, where some (partial) results on half-spaces are obtained under technical and restrictive assumptions on the nonlinear function $f$.

 While all these papers assume that $f(u) \geq 0$, 
 in the present work we are able to consider general genuinely changing-sign nonlinearities.

In particular, we assume that $f$ satisfies the following assumptions: 
\begin{equation*}
{(h_f)} \qquad \begin{cases}
f \in C^1([0,+\infty)), \quad f(0) \geq 0, & \\
\mathcal{N}_f:=\{t \in [0,+\infty) \ | \ f(t)=0\} \,\, {\text {is a discrete set.}} & 
\end{cases}
\end{equation*}

Since in our problem the right hand side
depends only on $u$, it is possible to introduce the following set
\begin{equation}\label{eq:zeroSet}
\mathcal{Z}_{f(u)}:= \{ x \in \R^N_+ \ | \ u(x) \in \mathcal{N}_f\}.
\end{equation}

\noindent We define a strip $\Sigma_\lambda$ of $\R^N_+$ as follows
$$\Sigma_\lambda:=\R^{N-1} \times (0,\lambda),$$
where $\lambda >0$. Now, under these assumptions and using these notations, we are ready to state our main result:

\begin{thm}\label{thm:halfspace}
Assume $N > 1$,  ${(2N+2)}/{(N+2)} < p < 2$ and let $u \in C^{1}_{loc}(\overline{\R^N_+})$  be a weak solution of \eqref{equation1} such that $\nabla u \in L^\infty(\Sigma_\lambda)$, for every strip $\Sigma_\lambda$ of $\R^N_+$. If $f$ fulfills  {\bf ($h_f$)}, then $u$ is monotone increasing w.r.t. the $y$-direction and
\begin{equation}\label{monotonicityxn}
\frac{\partial u}{\partial y} \geq 0 \quad \text{in} \quad \R^N_+.
\end{equation}
Moreover
\begin{equation}\label{monotonicityxnbis}
	\frac{\partial u}{\partial y} > 0 \quad \text{in} \quad \R^N_+ \setminus \mathcal{Z}_{f(u)}.
\end{equation}
\end{thm}

To get our main result, we use a weak comparison principle for quasilinear equations in strips \cite{FMSR} to start the moving plane procedure orthogonally to the boundary $\partial \R^N_+$. Then, by a delicate analysis based on the use of the techniques developed in \cite{DSCalcVar, DSJDE, EFMS, FMS, FMS3, FMSR} and by the translation invariance of the considered problem, we obtain the monotonicity of the solution w.r.t. the $y$-direction.

\section{Preliminary results and proof of Theorem \ref{thm:halfspace}}\label{sec: monotonicity}

The aim of this section is to prove Theorem \ref{thm:halfspace}. First of all we recall some results about comparison and maximum principles for quasilinear elliptic equations that will be used several times in the proof of Theorem \ref{thm:halfspace}, in the same spirit of \cite{EFMS}. To this end we give the definition of weak solution to problem \eqref{equation1}.

	\begin{defn}\label{weakform}
		Let $\Omega$ be an open set of $\R^N$, $N\geq 1$. We say that $u \in  C^1(\Omega)$ is a
		\emph{weak subsolution} to
		\begin{equation}\label{eq:probleadjunt}
		-\Delta_p u = f(u) \quad \text{in} \quad\Omega
		\end{equation}
		if
		\begin{equation}\label{problem1weaksubsol}
		\int_{\Omega} |\nabla u|^{p-2}(\nabla u, \nabla \varphi) \, dx  \, \leq
		\int_{\Omega} f(u) \varphi \, dx \quad \forall \varphi \in
		C_c^\infty(\Omega), \, \varphi \geq 0.
		\end{equation}
		Similarly, we say that $u \in C^1(\Omega)$ is a
		\emph{weak supersolution} to \eqref{eq:probleadjunt} if
		\begin{equation}\label{problem1weaksupersol}
		\int_{\Omega} |\nabla u|^{p-2} (\nabla u, \nabla \varphi) \, dx \, \geq \int_{\Omega} f(u) \varphi \, dx \quad \forall \varphi \in C_c^\infty(\Omega), \, \varphi \geq 0.
		\end{equation}
		Finally, we say that $u \in  C^1(\Omega)$ is a
		weak \emph{solution} of equation \eqref{eq:probleadjunt}, if
		\eqref{problem1weaksubsol} and \eqref{problem1weaksupersol}  hold.
	\end{defn}

The first theorem that we are going to recall is a strong comparison principle that holds for quasilinear elliptic equations in any bounded smooth domains $\Omega$ under the hypothesis that  the nonlinearity $f$ has a fixed sign. 

\begin{thm}[\cite{DSCalcVar}]\label{SCPLucioeDino}
Let $u,v\in C^1(\overline{\Omega})$ be two solutions to 
\begin{equation} \label{probSMCP}
	-\Delta_p w= f(w) \qquad  \text{in} \quad \Omega,
\end{equation} 
where $\Omega$ is a bounded  smooth connected domain  of $\mathbb{R}^N$ and $\frac{2N+2}{N+2}<p<+\infty$.
Assume that at least one of the following two conditions \textbf{$(f_u)$},\textbf{$(f_v)$} holds:
\begin{itemize}
	\item[\textbf{($f_u$):}] either
	\begin{equation}\label{fuasspos}
	f(u(x)) >0 \quad\mbox{in}\quad\overline{\Omega}
	\end{equation}
	or
	\begin{equation}\label{fuassneg}
	f(u(x)) <0 \quad\mbox{in}\quad\overline{\Omega};
	\end{equation}

	\item[\textbf{($f_v$):}] either
	\begin{equation}\label{fvasspos}
	f(v(x)) >0 \quad\mbox{in}\quad\overline{\Omega}
	\end{equation}
	or
	\begin{equation}\label{fvassneg}
	f(v(x)) <0 \quad\mbox{in}\quad\overline{\Omega}.
	\end{equation}
\end{itemize}
 {Suppose furthermore that}
\begin{equation}\label{hthCOMP:LAMB22II}
u\leq v\quad\mbox{in}\quad\Omega.
\end{equation}
Then $u\equiv v$ in $\Omega$ unless
\begin{equation}
u<v\quad \mbox{in}\quad\Omega.
\end{equation}
\end{thm}

For details regarding the proof of this result we refer to \cite{EFMS}. We remark that the classical strong comparison principle holds in any connected component $\Omega \setminus \mathcal{Z}$, where $\mathcal{Z}:=\{x \in \Omega \ | \  |\nabla u|=|\nabla v|=0\}$; Theorem \ref{SCPLucioeDino}, instead, is true without any a priori assumption on the set $\mathcal{Z}$. 

\smallskip

Now, let us recall that the linearized operator at a fixed solution $w$ of \eqref{probSMCP}, $L_w(v,\varphi)$ (with weight $\varrho=|\nabla u|^{p-2}$), is well defined for every smooth functions $v \in C^1(\hat \Omega)$ and $\varphi \in C^1_c(\hat \Omega)$, with $\hat \Omega := \Omega \setminus \{x \in \Omega \ | \ |\nabla w|=0\}$, by
\begin{equation}\label{eq:linearizzatoSCP}
\begin{split}
L_w(v, \varphi) \equiv & \int_{\hat \Omega} |\nabla w|^{p-2}(\nabla v, \nabla \varphi) \, dx +(p-2) \int_{\hat \Omega}|\nabla w|^{p-4}(\nabla w, \nabla v)(\nabla w, \nabla \varphi) \, dx\\
& - \int_{\hat \Omega} f'(w)v \varphi \, dx .
\end{split}
\end{equation}

As in the case of Theorem \ref{SCPLucioeDino}, it is possible to show a classical strong maximum principle for the linearized equation in any connected component of $\hat \Omega$. Although we defined the linearized operator \eqref{eq:linearizzatoSCP}, we need to define it at a fixed solution $w$ of \eqref{probSMCP} in the whole domain $\Omega$. We point out that under the hypothesis \eqref{fassposbis}
 and \eqref{fassnegbis}, thanks to \cite[Theorem 3.1]{SciunziCCM}, the weight $\varrho=|\nabla u|^{p-2} \in L^1(\Omega)$ if $p>(2N+2)/(N+2)$. Hence $L_w(v,\varphi)$, is well defined, for every $v$ and $\varphi$ in the weighted Sobolev space $H^{1,2}_\varrho(\Omega)$ with 
$\varrho=|\nabla w|^{p-2}$ by
\[
\begin{split}
	L_w(v, \varphi) \equiv & \int_\Omega |\nabla w|^{p-2}(\nabla v, \nabla \varphi) \, dx +(p-2) \int_\Omega|\nabla w|^{p-4}(\nabla w, \nabla v)(\nabla w, \nabla \varphi) \, dx\\
	& - \int_\Omega f'(w)v \varphi \, dx .
\end{split}
\]
We recall that, for $\varrho \in L^1(\Omega)$, the space $H^{1,2}_\varrho(\Omega)$ is define as the completion of $C^1(\Omega)$ (or $C^\infty(\Omega)$) with the norm
\begin{equation}\label{weightednorm}
	\|v\|_{H^{1,2}_\varrho}=\|v\|_{L^2(\Omega)}+\|\nabla v\|_{L^2(\Omega,\varrho)},
\end{equation}
where 
\[
\|\nabla v\|^2_{L^2(\Omega, \varrho)}:=\int_\Omega \varrho(x)|\nabla v(x)|^2 \, dx.
\]
Moreover, the space $H^{1,2}_{0,\varrho}(\Omega)$ is consequently defined as the closure of $C^1_c(\Omega)$ (or $C^\infty_c(\Omega)$), w.r.t. the norm \eqref{weightednorm}.
We point out that the linearized operator is also well defined if $v \in L^2(\Omega)$, $|\nabla w|^{p-2}\nabla w \in L^2(\Omega)$ and $\varphi \in W^{1,2}(\Omega)$.

Finally $v \in H^{1,2}_\rho(\Omega)$ is a weak solution of the linearized operator if
\begin{equation} \label{linearizedequation}
	L_w(v, \varphi)=0 \qquad {\forall \varphi \in H^{1,2}_{0,\rho}(\Omega)}.
\end{equation}

For future use we recall that, from \cite{DSJDE, SciunziNodea, SciunziCCM}, it follows  that $|\nabla u|^{p-2}\nabla u \in W^{1,2}_{loc} (\Omega)$ so that if $\varphi \in W^{1,2}_c(\Omega)$ (with compact support), the linearized operator $L_w(\partial_y u, \varphi)$ is still well defined. Under this assumption the following result holds true. 

\begin{thm}[\cite{DSCalcVar}]  \label{SMPlinearized}
	Let $u\in C^1(\Omega)$ be a solution to problem \eqref{probSMCP}, with $\frac{2N+2}{N+2}<p<+\infty$. Assume that either
\begin{equation}\label{fassposbis}
f(u(x)) >0 \quad\mbox{in}\quad\overline{\Omega}
\end{equation}
or
\begin{equation}\label{fassnegbis}
f(u(x)) <0 \quad\mbox{in}\quad\overline{\Omega}.
\end{equation}
If $\partial_y u \geq 0$ in $\Omega$, then  either $\partial_y u \equiv 0$ in $\Omega$ or $\partial_y u>0$ in $\Omega$.
\end{thm}

Finally, let us recall a weak comparison principle in narrow unbounded domains:
\begin{thm}[\cite{FMSR}]\label{compprinciplenarrow}
Let $1<p<2$ and $N>1$. Fix $\lambda_0 > 0$ and $L_0>0$. Consider $\lambda \in (0,\lambda_0]$, $\tau, \varepsilon > 0$ and set
$$\Sigma_{(\lambda,y_0)}:= \R^{N-1} \times \left(y_0-\frac{\lambda}{2}, y_0 + \frac{\lambda}{2} \right), \quad y_0 \geq \frac{\lambda}{2}, \quad y_0 + \frac{\lambda}{2} \leq \lambda_0.$$
Let $u,v \in C^{1}_{loc}(\overline{\Sigma}_{(\lambda, y_0)})$ such that $\|u\|_\infty + \|\nabla u\|_\infty \leq L_0$, $\|v\|_\infty + \|\nabla v\|_\infty \leq L_0$, $f$ fulfills \textbf{$(h_f)$} and
\begin{equation}\label{clevercomp}
\begin{cases}
-\Delta_p u \leq f(u) \quad & \text{in} \  \Sigma_{(\lambda, y_0)} \\
-\Delta_p v  \geq f(v) \quad & \text{in} \  \Sigma_{(\lambda, y_0)} \\
u \leq v \quad & \text{on} \  \partial \mathcal{S}_{(\tau, \varepsilon)}, \\
\end{cases}
\end{equation}
where the open set $\mathcal{S}_{(\tau, \varepsilon)} \subseteq \Sigma_{(\lambda, y_0)}$ is such that
$$\mathcal{S}_{(\tau, \varepsilon)} = \bigcup_{x' \in \R^{N-1}} I_{x'}^{\tau, \varepsilon},$$
	and the open set $I_{x'}^{\tau, \varepsilon} \subseteq \{x'\} \times (y_0-\frac{\lambda}{2}, y_0 + \frac{\lambda}{2} )$ has the form
$$I_{x'}^{\tau, \varepsilon} = A_{x'}^\tau \cup B_{x'}^\varepsilon, \; \text{with} \; |A_{x'}^\tau \cap B_{x'}^\varepsilon| = \emptyset$$
and, for $x'$ fixed, $A_{x'}^\tau, B_{x'}^\varepsilon \subset (y_0-\frac{\lambda}{2}, y_0 + \frac{\lambda}{2} )$ are measurable sets such that
$$|A^\tau_{x'}| \leq \tau \quad \text{and} \quad B_{x'}^\varepsilon \subseteq \{x_N \in \R \ | \ |\nabla u(x',x_N)| < \varepsilon, \ |\nabla v (x',x_N)| < \varepsilon \}.$$
Then there exist
$$\tau_0=\tau_0(N,p,\lambda_0,L_0)>0$$
and
$$\varepsilon_0=\varepsilon_0(N,p,\lambda_0,L_0)>0$$
such that, if $0 < \tau < \tau_0$ and $0 < \varepsilon < \varepsilon_0$, it follows that
$$u \leq v \quad \text{in} \,\,\mathcal{S}_{(\tau, \varepsilon)}.$$
\end{thm}

The proof of this result is contained in \cite[Theorem 1.6]{FMSR}, where the authors proved the same result for a more general class of operators and nonlinearities.

As remarked above, the purpose of this section consists in showing that all the non-trivial solutions $u$  to \eqref{equation1} are increasing in the $y$-direction. Since in our problem the right hand side depends only on $u$, we recall the definition \eqref{eq:zeroSet} of the set $\mathcal{Z}_{f(u)}$ introduced above
$$\mathcal{Z}_{f(u)}= \{ x \in \R^N_+ \ | \ u(x) \in \mathcal{N}_f\}.$$ 
Without any a priori assumption on the behaviour of $\nabla u$, the set $\mathcal{Z}_{f(u)}$ may be very wild (see e.g. Figure \ref{fg:wild}).

\begin{figure}[htbp]
	\centering
	\includegraphics[scale=.36]{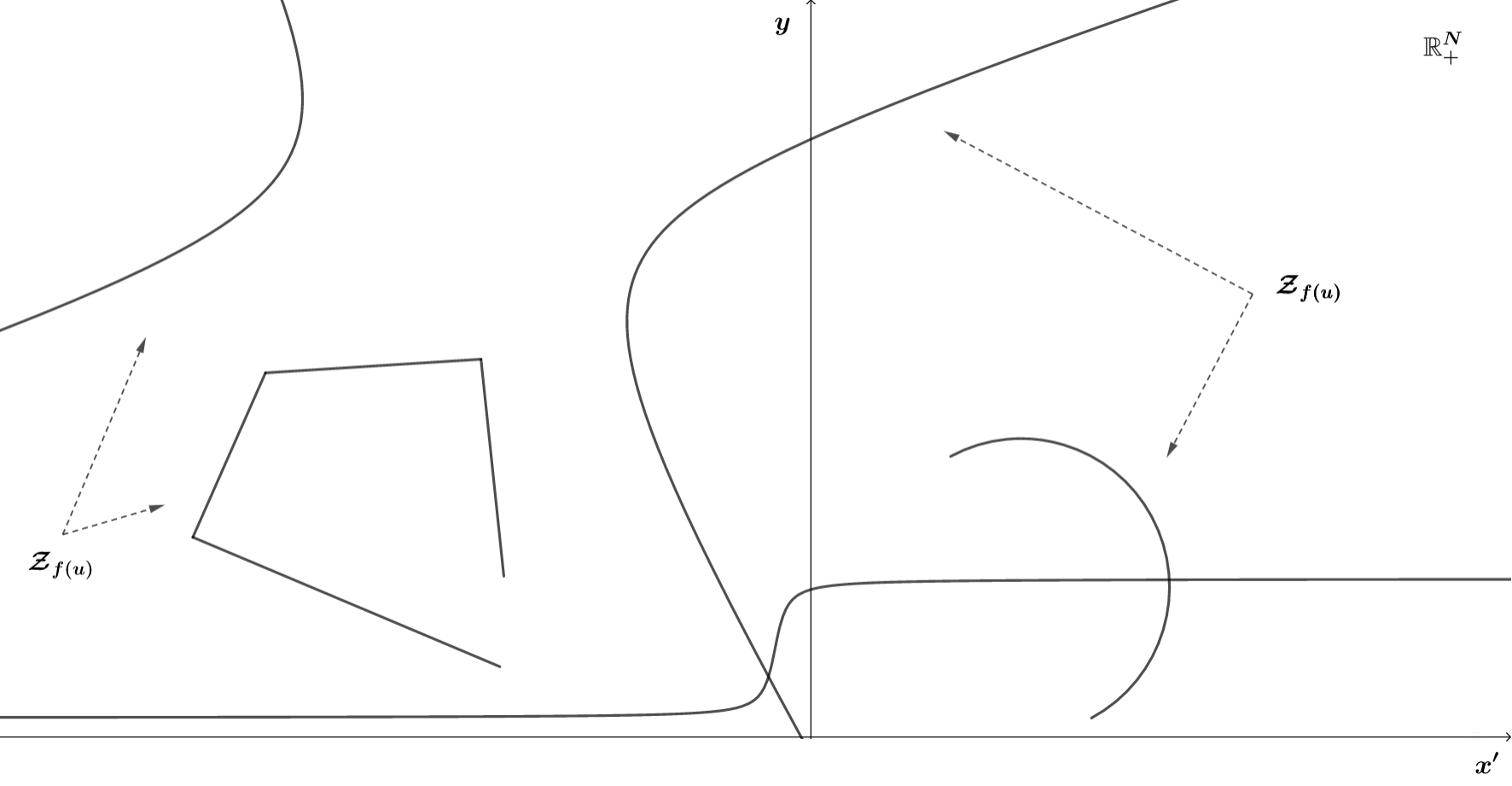}\\
	\caption{The  set $\mathcal{Z}_{f(u)}$ \label{fg:wild}}
\end{figure}

The proof of Theorem \ref{thm:halfspace} is based on a nontrivial modification of the moving plane method. Theorem \ref{thm:halfspace} will be proved at the end of this section using several preliminary results.

\begin{lem}\label{lem:utile}
	Let $\mathcal U$ a connected component of $\mathbb R^N_+ \setminus \mathcal{Z}_{f(u)}$ and let us assume that $\partial_y u\geq 0$ in $\mathcal U$.
	Then
	\[\partial_y u> 0\quad \text{in}\,\,\mathcal U.\]
\end{lem}

\begin{proof}
	Using Theorem \ref{SMPlinearized} we deduce that either $\partial_y u> 0$ in $\mathcal U$ or $\partial_y u\equiv 0$ in $\mathcal U$. By contradiction, let us assume that $\partial_y u \equiv 0$ in $\mathcal U$.  Take $P_0=(x'_0,y_0) \in \mathcal U$ and let us define
	\[s(t)=P_0-t y_0 \cdot \mathbf{e}_N, \quad t\in [0,1]\]
	and 
	\begin{equation}\label{eq:inffff}
	t_0=\sup \Big\{t\in[0,1]\,:\, s(\vartheta) \in \mathcal U,  \; \forall \vartheta \in [0,t)\Big\}.
	\end{equation}
	We note that the supremum in \eqref{eq:inffff} is well defined, since by definition the connected component $\mathcal U$ is an open set, and that $t_0 \in (0,1]$. 
	
	In the case $t_0=1$, we deduce that $u(P_0)=0$. Indeed, $u$ is constant on $s(t)$ for $t\in [0,1)$  (recall that  $\partial_y u \equiv 0$ in $\mathcal U$), $u=0$ on $\partial \R^N_+$. But this gives a contradiction with our assumption $u(x',y)>0$ in $\R^N_+$.
	
	In the case $t_0<1$, we deduce that $s(t_0)\in \mathcal{Z}_{f(u)}$ and therefore $f(u(s(t_0)))=f(u(P_0-t_0y))=0$. But $u$ is constant on $s(t)$ for $0 \leq t \leq t_0$, which implies that $f(u(P_0))=f(u(P_0-t_0 y))=0$, namely $P_0 \in \mathcal{Z}_{f(u)}$. The latter clearly contradicts the assumption $P_0 \in \mathcal U$. Therefore $\partial_y u> 0$ in $\mathcal U$ as desired.
	
\end{proof}	

	Some arguments used in the proofs of the results of this section, are based on a
	nontrivial modification of the moving plane method. Let us recall some notations. We
	define the unbounded strip $\Sigma_\lambda$ and the hyperplane
	$T_\lambda$ by
	\begin{equation} \label{eq:acquaesapone} \begin{split}
	\Sigma_\lambda := \{x \in \R^N_+ \ | \ 0<x_N < \lambda\},\ \quad  \quad
	T_\lambda &:= \partial \Sigma_\lambda = \{x \in \R^N_+ \ | \ x_N = \lambda\}
	\end{split} \end{equation}
	and the reflected function $u_\lambda(x)$ by
	$$u_\lambda (x) = u_\lambda (x',x_N):= u(x', 2\lambda-x_N) \quad \text{in}\,\, \mathbb R^N_+.$$
	We also define the critical set $\mathcal{Z}_{\nabla u}$ by
	\begin{equation}\label{eq:criticalset}
	\mathcal{Z}_{\nabla u}:=\{x \in \R^N_+ \ | \ \nabla u(x) = 0 \}.
	\end{equation}
	The first step in the proof of the monotonicity is to get a property concerning the local
	symmetry regions of the solution, namely any $\mathcal{C} \subseteq \Sigma_\lambda$ such that $u \equiv u_\lambda$ in $\mathcal{C}$.

	Having in mind these notations we are able to prove the following:

	\begin{prop}\label{connectedness}
		Under the assumption of Theorem \ref{thm:halfspace}, let us assume that $u$ is a solution to  \eqref{equation1}, such that
		\begin{itemize}
			\item[(i)] $u$ is monotone non-decreasing in $\Sigma_\lambda$
			
			\
			
			\text{and}
			
			\
			
			\item [(ii)] $u \leq u_\lambda$ in $\Sigma_\lambda$.
		\end{itemize}
		Then $u < u_\lambda$ in $\Sigma_\lambda \setminus \mathcal{Z}_{f(u)}$.
	\end{prop}
	
	\begin{proof}
		 Arguing by contradiction, let us assume that there exists $P_0
		=(x_0',y_0) \in \Sigma_\lambda \setminus \mathcal{Z}_{f(u)}$ such
		that $u(P_0)=u_\lambda(P_0)$. Let $\mathcal{U}_0$ the connected component of $\Sigma_\lambda \setminus \mathcal{Z}_{f(u)}$ containing $P_0$. By Theorem \ref{SCPLucioeDino}, since $u(P_0)=u_\lambda(P_0)$, we deduce that $\mathcal U_0$ is a local symmetry region, i.e. $u\equiv u_\lambda$ in $\mathcal{U}_0$.

        Since $\mathcal{U}_0$ is an open set of $\Sigma_\lambda \setminus \mathcal{Z}_{f(u)}$ (and also of $\R^N_+$) there exists $\rho_0=\rho_0(P_0)>0$ such that
		\begin{equation}\label{eq:pallaprimopunto}
		B_{\rho_0}(P_0) \subset \mathcal{U}_0.
		\end{equation}

		\begin{figure}[htbp]
			\centering
			\includegraphics[scale=.37]{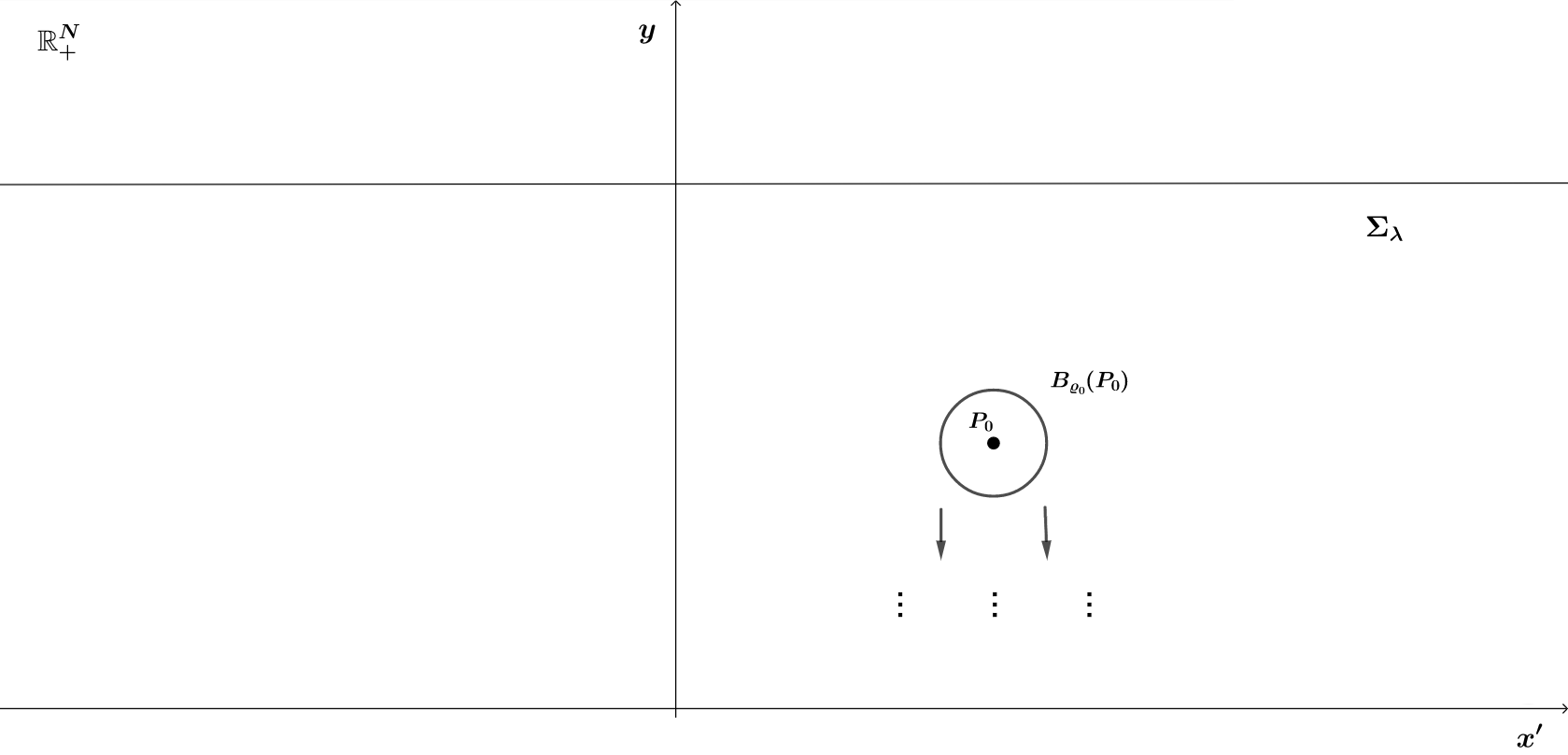}\\
			\caption{The slided ball $B_{\rho_0}(P_0)$ \label{fg:slided}}
		\end{figure}

		We can slide $B_{\rho_0}$ in $\mathcal U_0$, towards to $\partial \R^N_+$ in the $y$-direction and keeping its centre on the line $\{x'=x'_0\}$ (see Figure \ref{fg:slided}), until it touches  for the first time  $\partial \mathcal{U}_0$ at some point $z_0\in  \mathcal{Z}_{f(u)} \cap \R^N_+$. We observe that $z_0 \not \in \partial \R^N_+$, since otherwise $u \equiv u_\lambda \equiv 0$ contradicting the fact that $u$ is a positive solution. In Figure \ref{fg:1contact}, we show some possible examples of {\em first contact point} with the set $\mathcal{Z}_{f(u)}$.

		\begin{figure}[htbp]
			\centering
			\includegraphics[scale=.4]{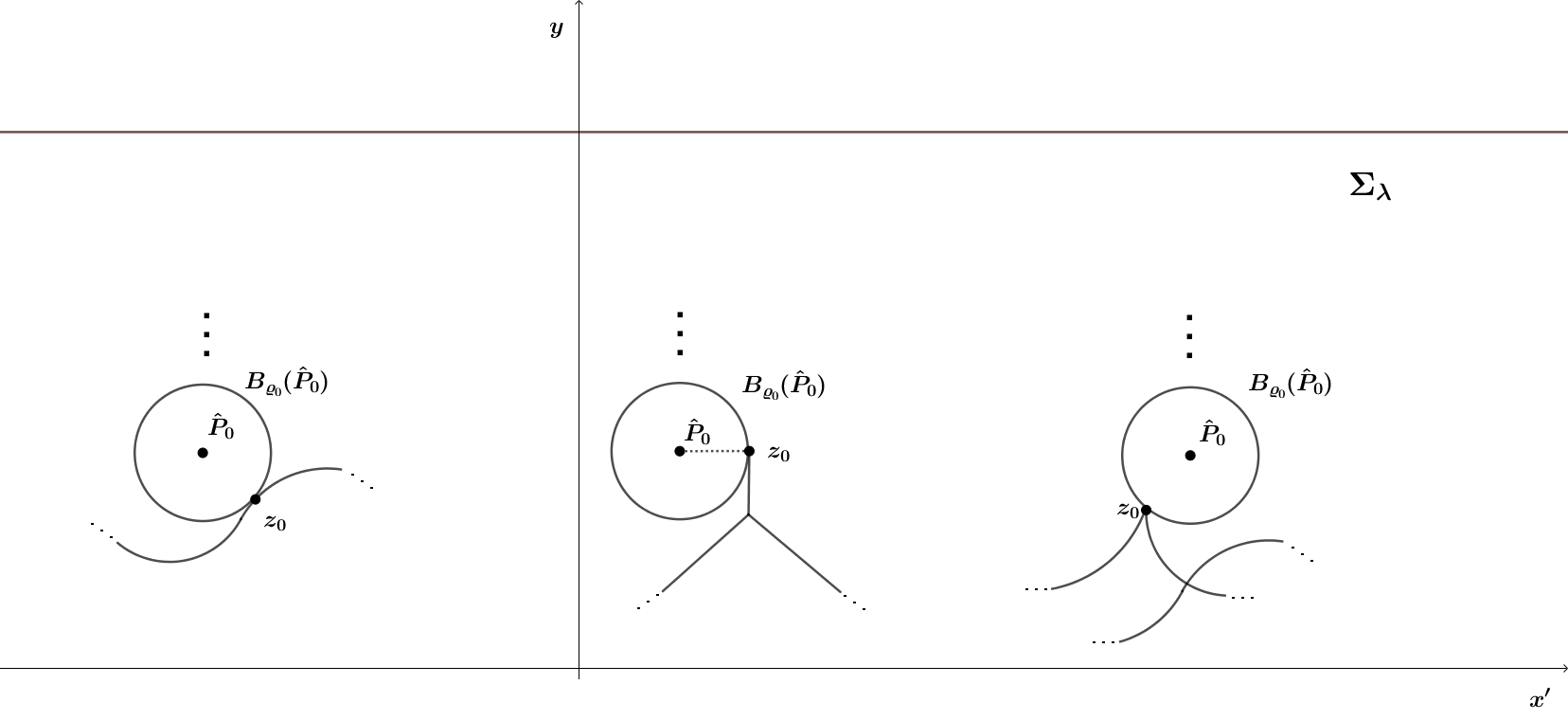}\\
			\caption{The first contact point $z_0$ \label{fg:1contact}}
		\end{figure}

		Now we consider the function
		$$w_0(x):=u(x)-u(z_0)$$
		and we observe that $w_0(x) \neq 0$ for every $x \in B_{\rho_0}(\hat P_0)$, where $\hat P_0$ is the new centre of the slided ball. In fact,  if
		this is not the case there would exist a point $\bar z \in  B_{\rho_0}(\hat P_0)$
		such that $w_0(\bar z)=0$, but this is in contradiction with the
		fact that $\mathcal{U}_0 \cap \mathcal{Z}_{f(u)}= \emptyset$.
		We
		have to distinguish two
		cases. Since $p < 2$ and $f$ is locally Lipschitz, we have that

		{\textbf{Case 1:} If $w_0(x) > 0$ in $B_{\rho_0}(\hat P_0)$, then
			\begin{equation}\nonumber
			\begin{cases}
			\Delta_p w_0 \leq C w_0^{p-1} & \quad \text{in} \ B_{\rho_0}(\hat P_0)\\ 
			w_0>0 & \quad \text{in} \ B_{\rho_0}(\hat P_0)\\
			w(z_0)=0 & \quad z_0 \in \partial B_{\rho_0}(\hat P_0),
			\end{cases}
			\end{equation}
			where $C$ is a positive constant.}
		
		{\textbf{Case 2:} If $w_0(x) < 0$ in $B_{\rho_0}(\hat P_0)$, setting $v_0 = - w_0$ we have
			\begin{equation}\nonumber
			\begin{cases}
			\Delta_p v_0 \leq C v_0^{p-1} & \quad \text{in} \ B_{\rho_0}(\hat P_0)\\ 
			v_0>0 & \quad \text{in} \ B_{\rho_0}(\hat P_0)
			\\
			v_0(z_0)=0 & \quad z_0 \in \partial B_{\rho_0}(\hat P_0),
			\end{cases}
			\end{equation}
			where $C$ is a positive constant.}
		
		In both cases, by the Hopf boundary lemma (see e.g. \cite{pucser, strongpucser, Vaz}), it
		follows that $|\nabla w(z_0)|=|\nabla u (z_0)| \neq 0$.

		Using the Implicit Function Theorem we deduce that the set $\{u=u(z_0)\}$
		is a  smooth manifold near $z_0$. Now we want to prove that
		\[
		\partial_y u (z_0) > 0
		\] 
		and actually that the set $\{u=u(z_0)\}$ is a graph in the $y$-direction near the point $z_0$. By our assumption we know that $\partial_y u (z_0):=u_y(z_0) \geq 0$.
		According to \cite{DSCalcVar, DSJDE} and \eqref{eq:linearizzatoSCP}, the linearized operator of \eqref{equation1} is well defined
		\begin{equation}\label{linearized}
		\begin{split}
		L_u(u_y, \varphi) \equiv & \int_{\Sigma_\lambda} |\nabla u|^{p-2} (\nabla u_y, \nabla \varphi) \, dx + (p-2) \int_{\Sigma_\lambda} |\nabla u|^{p-4} (\nabla u, \nabla u_y)(\nabla u, \nabla \varphi) \, dx\\
		&- \int_{\Sigma_\lambda} f'(u)u_y \varphi \, dx .
		\end{split}
		\end{equation}
		for every $\varphi \in C^1_c(\Sigma_\lambda)$. Moreover $u_y$  satisfies the
		linearized equation \eqref{linearizedequation}, i.e.
		\begin{equation}\label{linearizedEq}
		L_u(u_y, \varphi) =0 \quad \forall \varphi \in C^1_c(\Sigma_\lambda).
		\end{equation}
		Let us set $z_0=(z_0',y_0) \in \R^N_+$.  
		
		We have two possibilities: $u_y(z_0)=0$ or $u_y(z_0)>0$.

		{\textit{Claim:} We show that the case $u_y(z_0)=0$ is not possible. }

	If $u_y(z_0)=0$, then 
		\begin{equation} \label{eq:equiv0rho}
			u_y(x) = 0 \qquad \text{for every } x \in B_{\hat \rho}(z_0),
		\end{equation} 
	    for some positive ${\hat \rho}$; to prove this we use the fact that $|\nabla u(z_0)|\neq 0$, $u$ belongs to $\mathcal{C}^{1}$ and that it holds the classical strong maximum principle for the linearized operator (see e.g. Theorem 2.5 in \cite{EFMS} or \cite{pucser}).

		By construction there exists $0 < \varepsilon_1 < \hat{\rho}$ such that every point $z \in \mathcal{S}_1 := \{(z'_0,t) \in \mathcal{U}_0 : y_0  < t < y_0 + \varepsilon_1 \}$ has the following properties:
			\begin{enumerate}
			\item $z \in \overline{\mathcal{U}_0}$, since the ball is sliding along the segment $\mathcal{S}_1$;
				
			\item $z \not \in \partial \mathcal{U}_0$, since $z_0$ is the first contact point with $\partial \mathcal{U}_0$.

			\end{enumerate}
			That is, we have
			\begin{equation}\label{sliding}
			\mathcal{S}_1 \subset \overline{\mathcal{U}_0} \setminus \partial \mathcal{U}_0 = \mathcal{U}_0.
			\end{equation}
    		By construction we have that $ \overline{\mathcal{S}_1} \subset B_{\hat \rho}(z_0)$ and by \eqref{eq:equiv0rho} 
			$$u_y(z)=0 \quad \forall z \in \overline{\mathcal{S}_1}.$$
			Hence, we deduce that $u(z)=u(z_0)$ for every $z \in  \overline{\mathcal{S}_1}$; since $z_0 \in \mathcal{Z}_{f(u)}$ by our assumptions, we reach a contradiction with \eqref{sliding}.

		From what we have seen above,  we have $|\nabla u (z_0)| \neq 0$ and
		hence there exists a ball $B_r(z_0)$ where $|\nabla u (x)| \neq 0$
		for every $x \in B_r(z_0)$. By the classical strong comparison principle (see e.g. Theorem 2.3 in \cite{EFMS} or Theorem 1.4 in \cite{Damascelli} or \cite{pucser}) it follows that $u \equiv u_\lambda$ in
		$B_r(z_0)$ namely $u \equiv u_\lambda$ in a neighborhood of the point $z_0\in \partial \mathcal U_0$. Since $u_y (z_0)>0$ and $\mathcal{N}_f$ is discrete
		\[B_r(z_0)\cap \Big((\Sigma_\lambda \setminus \mathcal{Z}_{f(u)})\setminus \mathcal{U}_0\Big)\neq \emptyset\]
		and $u_y(x) >0$ in $B_r(z_0)$, as consequence,   the set $\{u=u(z_0)\}$ is a graph in the $y$-direction in a neighborhood  of  the point $z_0$.

		Now we have to distinguish two cases:
		
		\
		
		\textbf{Case 1:}  $u(z_0)=\min \Big [\mathcal{N}_f\setminus\{0\}\Big].$
			
		\noindent Define the sets 
		\[ \mathcal{C}_1:=\Big\{x=(x',y)\in \mathbb R^N_+ \,:\,  x'\in \Pi_0\left( B_r(z_0)\right) \quad \text{and}\quad u(x)<u(z_0)\Big \} \]
		\[ \mathcal{C}_2:= B_r(z_0) \cup \big( \Pi_0\left( B_r(z_0)\right) \times (0,y_0)\big),\]
		where $\Pi_0:\R^N \rightarrow \R^{N-1}$ is the canonical projection on the hyperplane $\{y=0\}$	and 
		\[\mathcal{C}=\mathcal{C}_1\cap \mathcal{C}_2.\]
		We observe that $\mathcal{C}$ is an open path-connected set (actually a deformed cylinder), see Figure \ref{fg:3lastcomponent}. Since $f(u(z_0))$ has the right sign, by Theorem \ref{SCPLucioeDino} it follows that $u\equiv u_\lambda$ in $\mathcal C$ and this is in  contradiction with the fact that $u(x',0)=0$ on $\partial \R^N_+$, while $u_\lambda(x',0)>0$.
		\begin{figure}[htbp]
			\centering
			\includegraphics[scale=.45]{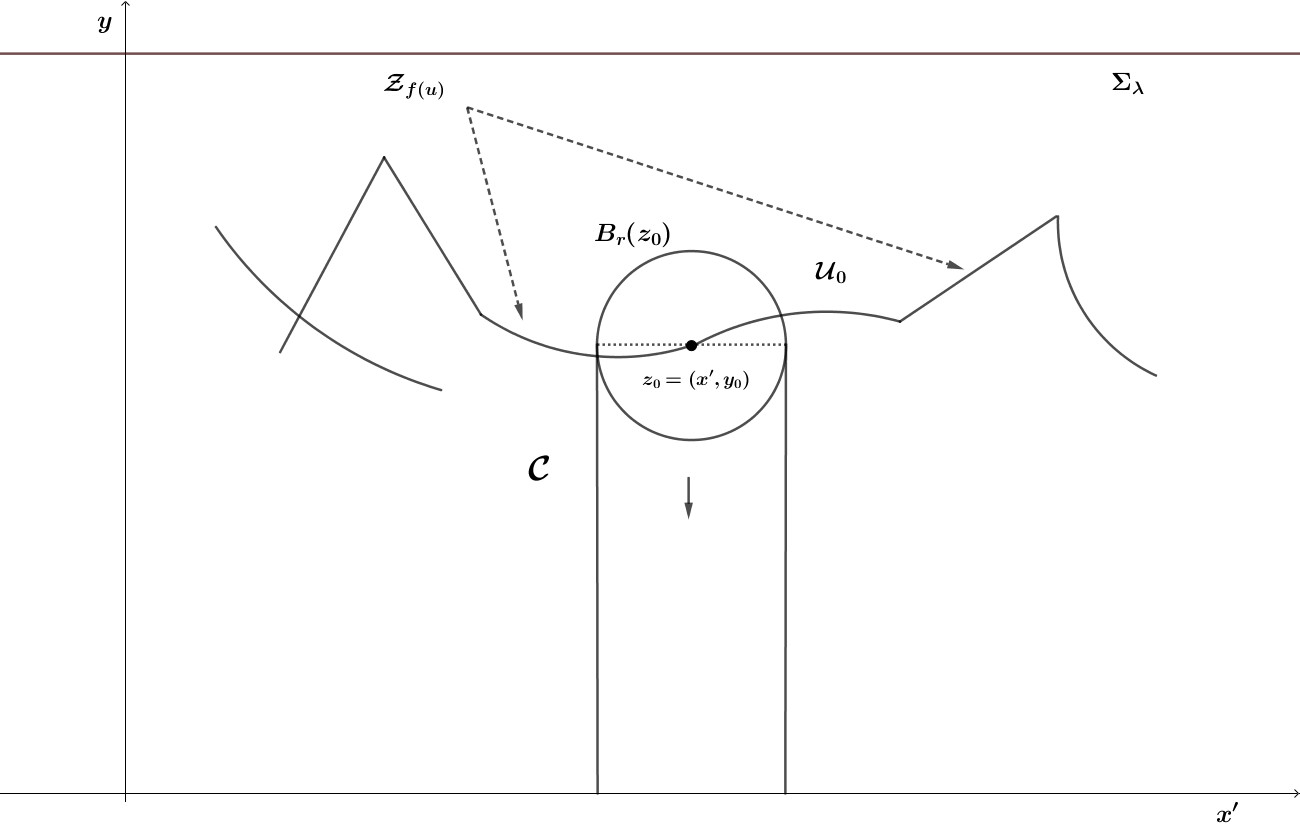}\\
			\caption{Case 1:  $u(z_0)=\min \Big [\mathcal{N}_f\setminus\{0\}\Big]$ \label{fg:3lastcomponent}}
		\end{figure}
			
		\ 
		
		\textbf{Case 2:} $u(z_0)>\min \Big [\mathcal{N}_f\setminus\{0\}\Big]$.
			
		\noindent In this case the open ball $B_r(z_0)$ must intersect another connected component  (i.e. $\not\equiv \mathcal{U}_0 $) of $\Sigma_\lambda \setminus \mathcal{Z}_{f(u)}$, such that  $u \equiv u_\lambda$ in  a such component, see Figure \ref{fg:3morecomponents}. Here we used the fact that near the (new) first contact point $z_1$, the corresponding level set is a graph in the $y$-direction and $u(z_1)<u(z_0)$, since $\mathcal{N}_f$ is a discrete set. Now, it is clear that repeating a finite number of times the argument leading to the existence of the touching point $z_0$, we can find a touching point $z_m$ such that \[u(z_m)=\min \Big [\mathcal{N}_f\setminus\{0\}\Big].\] The contradiction then follows exactly as in Case 1.
		\begin{figure}[htbp]
			\centering
			\includegraphics[scale=.25]{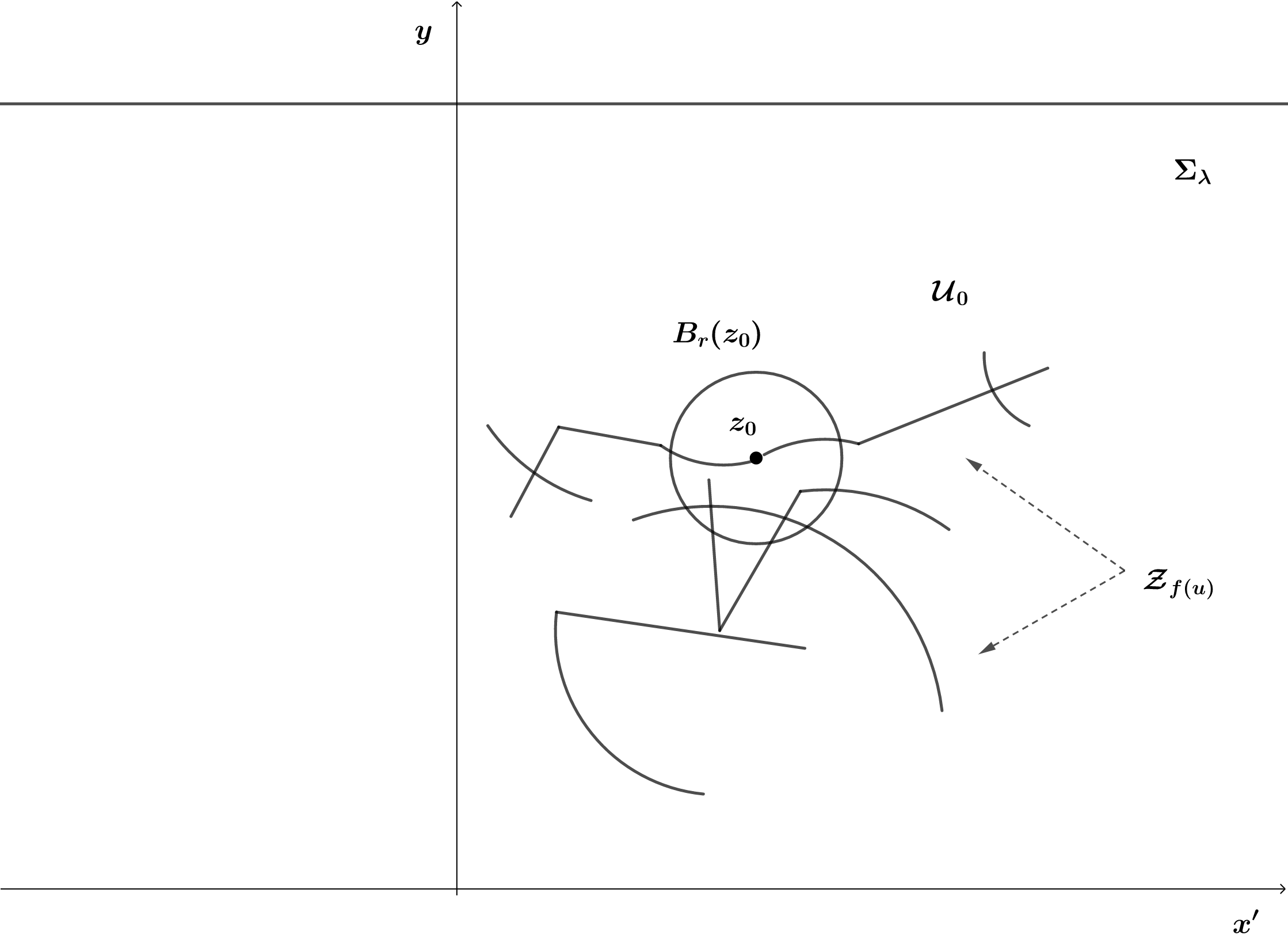}\\
			\caption{Case 2:  $u(z_0)>\min \Big [\mathcal{N}_f\setminus\{0\}\Big]$ \label{fg:3morecomponents}}
		\end{figure}

		Hence $u < u_\lambda$ in $\Sigma_\lambda \setminus
		\mathcal{Z}_{f(u)}$.
		
	\end{proof}

	With the notation introduced above, we set
	\begin{equation}\label{eq:miamnamgnam}
	\Lambda := \{\lambda \in \R \ | \ u \leq u_t \; \text{in} \ \Sigma_t, \; \forall t < \lambda\}.
	\end{equation}
	
	Note that, by Theorem \ref{compprinciplenarrow} (with $v=u_t)$, it follows
	that $\Lambda \neq \emptyset$, hence we can define
	\begin{equation}\label{eq:suplambda}
	\bar \lambda := \sup \Lambda.\end{equation}
	The proof of the fact that  $u(x',y)$ is monotone increasing in the $y$-direction in the half-space $\mathbb R^N_+$ is done once we show that $\bar	\lambda = + \infty$. To this end we assume by contradiction that $0 < \bar \lambda < + \infty$, and we prove a crucial result,
	which allows us to localize the support of $(u-u_{\bar \lambda})^+$.
	This localization, that we are going to obtain, will be useful to
	apply Theorem \ref{compprinciplenarrow}.
	
	\begin{prop}\label{localizationSupport}
		Under the same assumption of Theorem \ref{thm:halfspace}, let $u$ be a solution to \eqref{equation1}. Assume $0 < \bar	\lambda < + \infty$ \emph{(see \eqref{eq:suplambda})} and set
		\begin{equation}\nonumber
		W_\varepsilon:=(u-u_{\bar \lambda + \varepsilon}) \cdot \chi_{\{y \leq \bar \lambda + \varepsilon\}}
		\end{equation}\label{WepsilonBis}
        where $\varepsilon > 0$.
        
		Given $0 < \delta < \frac{\bar \lambda}{2}$ and $\kappa > 0$, there exists $\bar \varepsilon > 0$ such that for every $0 < \varepsilon < \bar \varepsilon$
		\begin{equation}\label{supportWepsilon}
		\text{\emph{supp}}\, W^+_\varepsilon \subset \{0 \leq y \leq \delta\} \cup
		\{ \bar \lambda - \delta \leq y \leq \bar \lambda + \varepsilon \} \cup
		\left( \bigcup_{x' \in \R^{N-1}} B_{x'}^{\kappa} \right),
		\end{equation}
	where $B_{x'}^{\kappa}$ is such that
	$$B^{\kappa}_{x'} \subseteq \{y \in (0,\bar \lambda + \varepsilon) \ : \ |\nabla u(x',y)|<\kappa, \; |\nabla u_{\bar \lambda + \varepsilon}(x',y)|< \kappa\}.$$
	\end{prop}
	
	\begin{proof} Since $\bar \lambda < + \infty$, the continuity of $u$ and $u_\lambda$, imply 
	$$u \leq u_{\bar \lambda} \quad \text{in} \; \Sigma_{\bar \lambda}.$$
	
	Assume by contradiction that \eqref{supportWepsilon} is false, so that there
		exist $0 < \delta < \frac{\bar \lambda}{2}$ and $\kappa > 0$ in such a way that, given any $\bar\varepsilon>0$, we find
		$0<\varepsilon \leq \bar\varepsilon$ so that there exists a
		corresponding $x_\varepsilon = (x'_\varepsilon, y_\varepsilon)$
		such that
		$$u(x'_\varepsilon, y_\varepsilon) \geq u_{\bar
			\lambda + \varepsilon} (x'_\varepsilon, y_\varepsilon),$$
		with $x_\varepsilon = (x'_\varepsilon, y_\varepsilon)$
		belonging to the set
		\[\{(x'_\varepsilon,y_\varepsilon)\in \mathbb R^N\,:\,\delta < y_\varepsilon < \bar \lambda - \delta \}\]
		such that $ |\nabla u(x_\varepsilon)| \geq \kappa$ (or possibly $|\nabla u_{\bar
			\lambda + \varepsilon} (x_\varepsilon)| \geq \kappa$).
		
		Taking $\bar\varepsilon = {1}/{n}$, then there exists
		$\varepsilon_n \leq 1/n$ going to zero, and a
		corresponding  sequence $$x_n = (x'_n, y_n) = (x'_{\varepsilon_n}, y_{\varepsilon_n})$$ such that
		$$u(x'_n, y_n) \geq
		u_{\bar \lambda + \varepsilon_n} (x'_n, y_n)$$
		with  $\delta < y_n < \bar \lambda - \delta$. Up to subsequences, let us
		assume that
		$$y_n \rightarrow \bar y \; \text{with} \;
		\delta \leq \bar y \leq \bar \lambda - \delta.$$
		
		Let us define
		$$\tilde u_n(x',y):=u(x'+x'_n, y)$$
		so that $\|\tilde u_n\|_\infty = \| u \|_\infty$. By standard
		regularity theory, see \cite{DB, Li, T}, we have that
		$$\|\tilde u_n\|_{C^{1,\alpha}_{loc}(\overline{\R^N_+})} \leq C,$$
	    for some $ \alpha \in(0,1)$. Therefore, by Ascoli's Theorem we have
		$$\tilde u_n \overset{C^{1}_{loc}(\overline{\R^N_+})}{\longrightarrow} \tilde u$$
		up to subsequences. By construction it follows that
		\begin{itemize}
			\item[(i)] $\tilde u > 0$ in $\R^N_+$, with $\tilde u(x',0)=0$ for every $x' \in \R^{N-1}$;
			\item[(ii)] $\tilde u \leq \tilde u_{\bar \lambda} \quad \text{in} \
			\Sigma_{\bar \lambda}$;
			
			\item[(iii)] $\tilde u(0,\bar y) = \tilde u_{\bar \lambda} (0, \bar y)$;
			
			\item[(iv)] $| \nabla \tilde u(0, \bar y)| \geq \kappa$.
		\end{itemize}
		
		We point out that, in (i), by construction  $\tilde u \geq 0$ in $\R^N_+$. As a consequence, by applying the strong maximum principle (see \cite{pucser, strongpucser, Vaz}) we get that $\tilde u >0$ in $\R^N_+$. Since $|\nabla \tilde{u} (0, \bar y)| \geq \kappa$ there exists $\rho >0$ and a ball $B_\rho(0,\bar y) \subset \Sigma_{\bar \lambda}$ such that $|\nabla \tilde u(x)| \neq 0$ for every $x \in B_\rho(0,\bar y)$. Now, if $\tilde{u}(0,\bar y) \in \mathcal{Z}_{f(u)}$, since $\tilde u$ is non constant  in  $B_\rho(0,\bar y)$, there exists  $P_0 \in B_\rho(0,\bar y)$ such that $\tilde u(P_0) \not \in \mathcal{Z}_{f(u)}$.
		By the classical strong comparison principle we get that
		\begin{equation}\label{disuguaglianza}
		\tilde{u} \equiv \tilde{u}_{\bar{\lambda}} \quad \text{in}\,\, B_\rho(0,\bar y).\end{equation}
		On the other hand, by Proposition \ref{connectedness} it follows that
		\begin{equation}\nonumber
		\tilde{u} < \tilde{u}_{\bar \lambda} \qquad \text{in} \,\,\Sigma_{\bar{\lambda}} \setminus \mathcal{Z}_{f(u)}.
		\end{equation}
		This gives a contradiction with \eqref{disuguaglianza}. Hence we have \eqref{supportWepsilon}.
		
	\end{proof}
	
	Now we are ready to prove the main result of the paper.
	
	\begin{proof}[Proof of Theorem \ref{thm:halfspace}]  Let us assume by contradiction that $0<\bar\lambda<+\infty$, see  \eqref{eq:suplambda}. Set $\lambda_0=\bar \lambda + 2$ and
	$$L_0:=\|u\|_{L^\infty(\{0\leq y \leq 2\bar \lambda + 10\})}+\|\nabla u\|_{L^\infty(\{0\leq y \leq 2\bar \lambda + 10\})}+1>0$$
	and take $\tau_0=\tau_0(N,p,\lambda_0,L_0)>0$ and $\varepsilon_0=\varepsilon_0(N,p,\lambda_0,L_0)>0$ as in Theorem \ref{compprinciplenarrow}. 
		
By  Proposition \ref{localizationSupport} we have that, given $0<\delta<\min\{\frac{\bar \lambda}{2},\frac{\tau_0}{4}\}$ and $0<\kappa<\varepsilon_0$, we find $\bar \varepsilon$ such that, for any $0<\varepsilon \leq \min\{\bar \varepsilon, \frac{\tau_0}{4},1\}$, it follows
\begin{equation}\nonumber
\text{supp}\, W^+_\varepsilon \subset \{0 \leq y \leq \delta\} \cup
\{ \bar \lambda - \delta \leq y \leq \bar \lambda + \varepsilon \} \cup
\left( \bigcup_{x' \in \R^{N-1}} B_{x'}^{\kappa} \right),
\end{equation}
where $B_{x'}^{\kappa}$ is such that
$$B^{\kappa}_{x'} \subseteq \{y \in (0,\bar \lambda + \varepsilon) \ : \ |\nabla u(x',y)|<\kappa, \; |\nabla u_{\bar \lambda + \varepsilon}(x',y)|< \kappa\}.$$

We claim $u \leq u_{\bar \lambda + \varepsilon}$ in $\Sigma_{\bar \lambda + \varepsilon}$. Indeed, let us assume that the open set
$$\mathcal{S}_{(2\delta+\varepsilon, \kappa)}:=\{x \in \Sigma_{\bar \lambda+\varepsilon} \ : \ u(x)-u_{\bar \lambda + \varepsilon}(x)>0\}$$
is not empty, then $u$ and $v=u_{\bar \lambda + \varepsilon}$ satisfy \eqref{clevercomp} with $\lambda=y_0=\bar \lambda + \varepsilon \ (<\lambda_0)$. Since by construction $2\delta + \varepsilon < \tau_0$ and $\kappa < \varepsilon_0$ we can apply Theorem \ref{compprinciplenarrow} to conclude that $u \leq u_{\bar \lambda + \varepsilon}$ on $\mathcal{S}_{(2\delta+\varepsilon, \kappa)}$. This contradicts the definition of $\mathcal{S}_{(2\delta+\varepsilon, \kappa)}$. Hence $\mathcal{S}_{(2\delta+\varepsilon, \kappa)}=\emptyset$.

		    

The latter proves that $u \leq u_{\bar \lambda + \varepsilon}$ in $\Sigma_{\bar \lambda + \varepsilon}$, which contradicts the definition of $\bar \lambda$ and consequently we deduce that $\bar \lambda=+\infty$. This implies the monotonicity of $u$, that is 
$$\frac{\partial u}{\partial y} \geq 0 \quad \text{in} \, \R^N_+.$$ 
Moreover, an application of Lemma \ref{lem:utile} yields 
$$\frac{\partial u}{\partial y} > 0 \quad \text{in} \, \R^N_+ \setminus \mathcal{Z}_{f(u)}.$$

\end{proof}


\begin{thebibliography}{99}



\bibitem{BCN0} {\sc H. Berestycki, L. A. Caffarelli and L. Nirenberg}.
\newblock Inequalities for second-order elliptic equations with applications to unbounded domains. I. A celebration of John F. Nash, Jr.
\newblock {\em Duke Math. J.} 81(2), 1996, pp. 467--494.

\bibitem{BCN1} {\sc H. Berestycki, L. A. Caffarelli and L. Nirenberg}.
\newblock Monotonicity for Elliptic Equations in Unbounded Lipschitz Domains.
\newblock {\em Comm. Pure Appl. Math.} 50(11), 1997, pp. 1089--1111.

\bibitem{BCN2} {\sc H.~Berestycki, L.~Caffarelli and L.~Nirenberg}.
\newblock Further qualitative properties for elliptic equations in unbounded domains. Dedicated to Ennio De Giorgi.
\newblock {\em Ann. Scuola Norm. Sup. Pisa Cl. Sci. $(4)$}, 25(1-2), 1997, pp. 69--94.


\bibitem{charro} {\sc F. Charro, L. Montoro and B. Sciunzi}.
\newblock Monotonicity of solutions of fully nonlinear uniformly elliptic equations in the half-plane.
\newblock {\em J. Differential Equations},  251(6), 2011, pp. 1562--1579.

\bibitem{Damascelli}  {\sc L. Damascelli}.
\newblock Comparison theorems for some quasilinear degenerate elliptic operators and applications to symmetry and monotonicity results.
\newblock {\em Ann. Inst. H. Poincar\'{e} Anal. Non Lin\'{e}aire}, 15(4), 1998, pp. 493--516.


\bibitem{DSCalcVar}  {\sc L. Damascelli and B. Sciunzi}.
\newblock Harnack inequalities, maximum and comparison principles.
and regularity of positive solutions of m-Laplace equations.
\newblock {\em Calc. Var. Partial Differential Equations}, 25(2), 2006, pp. 139-159.

\bibitem{DSJDE} {\sc L. Damascelli and B. Sciunzi}. 
\newblock Regularity, monotonicity and symmetry of positive solutions of $m$-Laplace equations. 
\newblock {\em J. Differential Equations}, 206(2), 2004,  pp. 483--515.


\bibitem{DanDuEf} {\sc E. N. Dancer, Y. Du and M. Efendiev}.
\newblock Quasilinear elliptic equations on half- and quarter-spaces.
\newblock {\em Adv. Nonlinear Stud.} 13(1), 2013, pp. 115--136.

\bibitem{DB} {\sc E. Di Benedetto}. 
\newblock $C^{1+\alpha}$ local regularity of weak solutions of degenerate elliptic equations. \newblock {\em Nonlinear Anal.} 7(8),  1983, 827--850.

\bibitem{DuGuo} {\sc Y. Du and Z. Guo}.
\newblock Symmetry for elliptic equations in a half-space without strong maximum principle.
\newblock \emph{Proc. Roy. Soc. Edinburgh Sect. A}, 134(2), 2004, pp. 259--269.

\bibitem{DupaigneFarina} {\sc L. Dupaigne and A. Farina}.
\newblock Classification and Liouville-type theorems for semilinear elliptic equations in unbounded domains.
\newblock To appear in {\em Anal. PDE}. Preprint: \textbf{arxiv.org/pdf/1912.11639.pdf}.

\bibitem{EFMS} {\sc F. Esposito, A. Farina, L. Montoro and B. Sciunzi}.
\newblock On the Gibbons' conjecture for equations involving the $p$-Laplacian. 
\newblock {\em Math. Ann.}, 2020, DOI number: 10.1007/s00208-020-02065-7.

\bibitem{Farina1} {\sc A. Farina}.
\newblock Some results about semilinear elliptic problems on half-spaces.
\newblock {\em Math. Eng.} 2(4), 2020, pp. 709--721.


\bibitem{FMS} {\sc A. Farina, L. Montoro and B. Sciunzi}.
\newblock Monotonicity and one-dimensional symmetry for solutions of $-\Delta_p u=f(u)$ in half-spaces.
\newblock {\em Calc. Var. Partial Differential Equations}, 43(1--2), 2012, pp. 123--145.

\bibitem{FMS3} {\sc A.~Farina, L.~Montoro and B.~Sciunzi}.
\newblock  Monotonicity of solutions of quasilinear degenerate elliptic equations in half-spaces.
\newblock {\em Math. Ann.}, 357(3), 2013, pp. 855--893.

\bibitem{FMSR} {\sc A. Farina, L. Montoro, G. Riey and B. Sciunzi}.
\newblock Monotonicity of solutions to quasilinear problems with a first-order term in half-spaces.
\newblock {\em Ann. Inst. H. Poincar\'e Anal. Non Lin\'eaire}, 32(1), 2015, pp. 1--22.

\bibitem{FMS_ans} {\sc A. Farina, L. Montoro and B. Sciunzi}.
\newblock  Monotonicity in half-space of positive solutions to $-\Delta_p u=f(u)$ in the case $p>2$.  
\newblock {\em Ann. Sc. Norm. Super. Pisa Cl. Sci.}, 17(5), 2017, pp. 1207--1229.

\bibitem{farsciu}{\sc A. Farina and B. Sciunzi}.
\newblock Qualitative properties and classification of nonnegative solutions to $-\Delta u = f(u)$ in unbounded domains when $f(0) < 0$.
\newblock {\em Rev. Mat. Iberoam.} 32(4), 2016, pp. 1311--1330.

\bibitem{farsciubis}{\sc A. Farina and B. Sciunzi}.
\newblock Monotonicity and symmetry of nonnegative solutions to $-\Delta u = f(u)$ in half-planes and strips. 
\newblock {\em Adv. Nonlinear Stud.} 17(2), 2017, pp. 297--310.


\bibitem{Li} {\sc G.M. Lieberman}.
\newblock Boundary regularity for solutions of degenerate elliptic equations.
\newblock {\em Nonlinear Anal.}, 12(11), 1988, pp. 1203--1219.

\bibitem{pucser} {\sc P. Pucci and J. Serrin}.
\newblock The maximum principle.
\newblock {\em Progress in Nonlinear Differential Equations and their Applications, 73. Birkh\"{a}user Verlag, Basel}, 2007.

\bibitem{strongpucser} {\sc P. Pucci and J. Serrin}. 
\newblock The strong maximum principle revisited. 
\newblock {\em J. Differential Equations}, 196(1), 2004, pp. 1--66.

\bibitem{QS} {\sc A. Quaas and B. Sirakov}.
\newblock Existence results for nonproper elliptic equations involving
the {P}ucci operator.
\newblock {\em Comm. Partial Differential Equations} 31(7-9), 2006, pp. 987--1003.

\bibitem{SciunziNodea} {\sc B. Sciunzi}.
\newblock Some results on the qualitative properties of positive solutions of quasilinear
elliptic equations.
\newblock {\em NoDEA Nonlinear Differential Equations Appl.}, 14(3–4), 2007, pp. 315--334.

\bibitem{SciunziCCM} {\sc B. Sciunzi}.
\newblock Regularity and comparison principles for $p$-Laplace equations with vanishing source term.
\newblock {\em Comm. Cont. Math.}, 16(6), 2014,  20 pp.




\bibitem{T} {\sc P. Tolksdorf}.
\newblock Regularity for a more general class of quasilinear elliptic
equations.
\newblock {\em J. Differential Equations},  51(1), 1984,  pp. 126--150.

\bibitem{Vaz} {\sc J.L. V\'azquez}.
\newblock A strong maximum principle for some quasilinear elliptic
equations.
\newblock {\em Appl. Math. Optim.}, 12(3), 1984, pp. 191--202.
\end{thebibliography}
\end{document}